\documentclass[11pt]{amsart}
\usepackage{tikz,amsthm,amsmath,amstext,amssymb,amscd,epsfig,euscript, mathrsfs, dsfont,pspicture,multicol,graphpap,graphics,graphicx,times,enumerate,subfig,sidecap,wrapfig,color}

\usepackage{amsmath}
\usepackage{amssymb}

\newcommand{\C}{{\mathbb C}}       
\newcommand{\R}{{\mathbb R}}       
\newcommand{\N}{{\mathbb N}}       %
\newcommand{\Z}{{\mathbb Z}}       

\newcommand{\FF}{{\mathcal F}}
\newcommand{\HH}{{\mathcal H}}

\newcommand{\GZ}{{\mathcal G}}

\newcommand{\SSS}{{\mathcal S}}

\newcommand{\diam}{\mathop{\rm diam}}
\newcommand{\dist}{{\rm dist}}

\newcommand{\ve}{{\varepsilon}}
\newcommand{\vv}{{\vspace{2mm}}}
\newcommand{\vvv}{\vspace{4mm}}

\newcommand{\noi}{\noindent}

\newcommand{\HD}{{\mathsf{HD}}}
\newcommand{\LD}{{\mathsf{LD}}}

\newcommand{\tree}{{\mathsf{Tree}}}

\def\XXint#1#2#3{{\setbox0=\hbox{$#1{#2#3}{\int}$ }
\vcenter{\hbox{$#2#3$ }}\kern-.58\wd0}}

\newtheorem{theorem}{Theorem}[section]
\newtheorem{lemma}[theorem]{Lemma}

\newtheorem{corollary}[theorem]{Corollary}

\newtheorem{proposition}[theorem]{Proposition}

\newtheorem*{lemma*}{Lemma}
\newtheorem*{theorem*}{Theorem}

\theoremstyle{definition}

\theoremstyle{remark}
\newtheorem{rem}[theorem]{\bf Remark}

\numberwithin{equation}{section}

\newcommand{\RRem}{\begin{rem}}
\newcommand{\erem}{\end{rem}}

\makeatletter
\def\@tocline#1#2#3#4#5#6#7{\relax
  \ifnum #1>\c@tocdepth 
  \else
    \par \addpenalty\@secpenalty\addvspace{#2}%
    \begingroup \hyphenpenalty\@M
    \@ifempty{#4}{%
      \@tempdima\csname r@tocindent\number#1\endcsname\relax
    }{%
      \@tempdima#4\relax
    }%
    \parindent\z@ \leftskip#3\relax \advance\leftskip\@tempdima\relax
    \rightskip\@pnumwidth plus4em \parfillskip-\@pnumwidth
    #5\leavevmode\hskip-\@tempdima
      \ifcase #1
       \or\or \hskip 1em \or \hskip 2em \else \hskip 3em \fi%
      #6\nobreak\relax
    \dotfill\hbox to\@pnumwidth{\@tocpagenum{#7}}\par
    \nobreak
    \endgroup
  \fi}
\makeatother

\def\cE{{\mathscr{E}}}
\def\cF{{\mathscr{F}}}

\def\cR{{\mathscr{R}}}

\def\cW{{\mathscr{W}}}

\begin{document}

\title[Carleson measure estimates for bounded harmonic functions] {Carleson measure estimates and $\boldsymbol\ve$-approximation for bounded harmonic functions, without Ahlfors regularity assumptions}

\author[Garnett]{John Garnett}

\address{John Garnett\\
Department of Mathematics, 6363 Mathematical Sciences Building \\University of California
at Los Angeles, 520 Portola Plaza, Los Angeles, California 90095-1555.
}
\email{jbg@math.ucla.edu}

\thanks{J.G. was supported by NSF Grant DMS 1217239.}

\begin{abstract}

Let $\Omega$ be a domain in $\R^{d+1}$, $d \geq 1$.  
In [HMM2] and [GMT] it was proved that if $\Omega$ satisfies a corkscrew condition and if $\partial \Omega$ is
$d$-Ahlfors 
regular, i.e.   Hausdorff measure $\HH^d(B(x,r) \cap \partial \Omega) \sim r^d$ for 
all $x \in \partial \Omega$ and $0 < r < \diam(\partial \Omega)$, then $\partial \Omega$ is uniformly  rectifiable if and only if 
\vv

(a)   a square function Carleson measure estimate holds for all bounded harmonic functions on $\Omega$ 

\vv

\noi or

\vv

(b) an $\ve$-approximation property holds for all such functions and all $0 < \ve < 1.$  

\vv
\noindent  Here we explore (a) and (b) when $\partial \Omega$  is not required to be Ahlfors regular. 
We first observe  that (a) and (b) hold for any domain $\Omega$ for which there exists a domain $\widetilde \Omega \subset \Omega$ such that  
$\partial \Omega \subset \partial \widetilde \Omega$ and $\partial \widetilde \Omega$ is uniformly rectifiable. 
 We then  assume  $\Omega$ satisfies a corkscrew condition and
$\partial \Omega$ satisfies a capacity density condition. 
Under these assumptions, we prove conversely that if {\rm (a)} or {\rm (b)} holds for $\Omega$ then such a domain  $\widetilde \Omega \supset \Omega$ exists and  
 give two further characterizations of domains where  (a) or (b) holds.  One is   
that  harmonic measure for $\Omega$  satisfies a Carleson packing condition with respect to diameters similar to the corona decomposition comparing harmonic measures to $\HH^d$  proved equivalent to  uniform rectifiability in  [GMT]. The second characterization is reminiscent of the 
Carleson measure description of $H^{\infty}$ interpolating sequences in the unit disc.

\end{abstract}

\maketitle

\tableofcontents

\vvv

\noi {\bf Acknowledgements.} I thank J. Azzam, J. M. Martell,  S. Mayboroda, M. Mourgoglou and 
X. Tolsa for helpful conversations.



\section{Introduction}

Let $\Omega \subset \R^{d+1}$ be open.   
We say bounded harmonic functions on $\Omega$ satisfy a {\it {Carleson measure estimate}} if there is a constant 
$C>0$ such that 
\begin{equation}\label{(1.1)}
\frac{1}{r^d} \int_{B(x,r) \cap \Omega} |\nabla u(x)|^2\,\dist(x,\partial\Omega)\,dx\leq C\|u\|^2_{L^\infty(\Omega)}.
\end{equation}

\vv
\noi whenever $x\in \partial \Omega, 0 < r < \diam(\Omega)$, and $u$ is a bounded harmonic function on $\Omega.$ 
It is a famous result of C. Fefferman [FS] that 
(1.1)
holds for the half spaces $\R^{d+1}_{+}.$

\vv

If $u$ is a bounded harmonic function on $\Omega$ and 
if $0 < \ve  < 1$, we say that
$u$ is 
  $\ve$-\textit{approximable} 
if there exists $g \in W^{1,1}_{\rm {loc}}(\Omega)$
and $C>0$ such that 

\begin{equation}\label{(1.2)}
\|u - g\|_{L^{\infty}(\Omega)} < \varepsilon
\end{equation}        
and for all $x \in \partial\Omega$ and all $r > 0$

\begin{equation}\label{(1.3)}
\frac{1}{r^d} \int_{B(x,r) \cap \Omega} |\nabla g(y)| \,dy \leq C.
\end{equation}

\noi It is clear by  normal families that  (1.2) and (1.3) then hold for every bounded harmonic function on $\Omega$ with constant $C = C_{\ve}$ depending only on  $\ve$ and $\Omega$. It is also clear that after local mollifications (1.2) and (1.3) will hold with $g \in C^{\infty}(\Omega)$; see [Ga], page 347.

\vv 
The notion of $\ve$-approximation was introduced by Varopoulos in [Va1] and [Va2] in works 
on 
corona problems   and $H^1-BMO$ duality.   Chapter VIII of [Ga] gave  a proof for all 
$\ve >0$ on the upper half plane and Dahlberg [Dal] extended the proof to Lipschitz domains,  using his work connecting  
square functions to maximal functions.  Later  Kenig, Koch, Pipher and Toro [KKoPT] applied 
$\ve$-approximation to more general elliptic boundary value problems and proved to that on any Lipschitz domain 
 elliptic harmonic measure is $A_{\infty}$ equivalent to boundary surface measure.  Further connections   
 between $\ve$-approximation,  Carleson measure  estimates,
square functions, maximal  functions, and $A_{\infty}$ conditions for elliptic measures have been  obtained on Lipschitz domains 
by several authors, including [DKP], [JK], [HKMP], [KKiPT]  
 and [Pi],  
and then on domains with Ahlfors regular boundaries by [AGMT], [HLMN], [HM1], [HM2] and [HMM2],  and most recently by [Az], [AHMMT], [BH], [H], [HMM3], [HMMTZ1] and [HMMTZ2]. 

\vv
The papers  [HMM2] and [GMT]  connect $\ve$-approximation and Carleson measures to rectifiability in domains with Ahlfors regular boundaries.  To explain their work we give three definitions:
The open set  $\Omega \subset \R^{n}$ satisfies a {\emph {corkscrew condition}} if 
there exists a constant $\alpha$ such that whenever $x \in \partial \Omega$ and $0 < r < {\rm {diam}}(\Omega)$ there exists ball $B(p,\alpha r)$ so that

\begin{equation}\label{(1.4)}
B(p,\alpha r) \subset \Omega \cap B(x,r).
\end{equation} 

\noi  If $\Omega$ is a connected open set with the corkscrew condition we say $\Omega$ is a {\emph {corkscrew domain}}. 
For $n >  d \geq 0$, a set $E\subset\R^{n}$ is called $d$-{\textit {Ahlfors regular}} (or simply Ahlfors regular if $d$ is clear from the context) if there exists a
constant $c>0$ such that for all $x \in E$ and $0 < r < \diam(E),$

\begin{equation} \label{(1.5)}
c^{-1}r^d\leq \HH^d(B(x,r)\cap E)\leq c\,r^d
\end{equation}
where $\HH^d$ denotes $d$-dimensional Hausdorff measure. 
When $1 \leq d < n$ is an integer, the set $E\subset\R^{n}$ is  {\textit {uniformly}}  $d$-{\textit {rectifiable}} if it is 
$d$-Ahlfors regular and 
there exist constants $c$ and  $M >0$ such that for all $x \in E$ and all $0<r\leq \diam(E)$ 
there is a Lipschitz mapping $g$ from the ball $B(0,r) \subset \R^{d}$ to $\R^n$ such that
$\text{Lip}(g) \leq M$ and

\begin{equation}\label{(1.6)}
\HH^d (E\cap B(x,r)\cap g(B_{d}(0,r)))\geq c r^d.
\end{equation}

\vv
\noindent Uniform rectifiability is a quantitative version of rectifiability.  It was introduced in the pioneering works [DS1] and [DS2] and David and Semmes who proved  $(n-1)$-uniform rectifiablity is     
a  geometric condition under which all singular integrals with sufficiently smooth odd kernels are 
bounded in $L^2(\partial \Omega)$. Later [MMV] and [NToV] proved conversely that the	 $L^2$ boundedness of  the Cauchy integral or the Riesz transforms on an Ahlfors regular boundary $\partial \Omega$ implies $\partial \Omega$ is $(n-1)$ uniformly rectifiable.  

\vv

The papers   [HMM2] and [GMT] prove that 
if  $\Omega\subset\R^{d+1}$,  $d\geq1$, is a corkscrew domain and $\partial \Omega$ is $d$-Ahlfors regular, then the following are equivalent:
\begin{itemize}
\item[(a)] All bounded harmonic functions on $\Omega$ satisfy the Carleson measure estimate (1.1).
\vv
\item[(b)] Every bounded harmonic function on $\Omega$ is $\varepsilon$-approximable for all
$0 < \varepsilon < 1.$ 
\vv
\item[(c)] $\partial\Omega$ is uniformly $d$-rectifiable.
\end{itemize}

\vv

\noi In fact,  [HMM2] proved {\rm (c)} implies {\rm (a)} and {\rm (b)} and [GMT] proved the converse statements.
Here our goal is to understand  the conditions  {\rm (a)} and {\rm  (b)}   when  $\partial \Omega$ is not
 necessarily  Ahlfors regular.  To state our results we need two more definitions.
We will usually assume $\Omega$  satisfies  a \textit {capacity density condition}: there is $\beta > 0$ such that for all $x \in \partial \Omega$ and $r \leq \diam(\Omega),$

\begin{equation}\label{(1.7)}
{\rm{Cap}}(B(x,r) \setminus  \Omega)\geq \left\{\begin{array}{ll}
\beta r &\quad \mbox{if $d+1 =2$,}\\
\\
\beta r^{d-1} &\quad \mbox{if $d+1 \geq 3$.}
\end{array}
\right.
\end{equation}

\noi where {\rm {Cap}} is Newtonian capacity when $d+1 \geq 3$ and logarithmic capacity when $d+1 =2.$
 If $\Omega$ satisfies (1.7)  every point of $\partial \Omega$ is regular for the Dirichlet problem, so that  for each $p \in \Omega$ there exists an unique Borel probability $\omega_p = \omega(p,.,\Omega)$ on $\partial \Omega$ such that 

\begin{equation}\label{(1.8)}
u(p) = \int_{\partial \Omega} u(x) d\omega(p,x,\Omega)
\end{equation}

\noindent  if $u$ is continuous $\overline \Omega$ and harmonic on $\Omega$.   Moreover, if $u(x)$ is continuous  on $\partial \Omega$, (1.8) defines a function harmonic on $\Omega$
which  continuously extends $u$ from $\partial \Omega$ to $\overline \Omega.$ Since $\Omega$ is connected it follows from Harnack's inequality that for all $p,q \in \Omega$ there is a constant $C_{p,q} = C_{p,q}(\Omega)$ such that $\omega_p \leq 
C_{p,q} \omega_q.$ 
The measure $\omega_p$ is called {\emph {harmonic measure for}} $p$.

\vv

\begin{theorem}\label{thm1}
Let $\Omega\subset\R^{d+1}$, $d\geq 1$, be a domain.  

\vv
\noi {\bf {A.}} ~If there exists a domain $\widetilde \Omega$ such that

\begin{equation}\label{(1.9)}
\widetilde \Omega \subset \Omega ~~~ and  ~~~ \partial \Omega \subset \partial \widetilde\Omega,
\end{equation}

\noi  and $\partial \tilde \Omega$ is uniformly rectifiable, then {\rm (a)} and {\rm (b)} hold for $\Omega.$ 

\vv

\noi {\bf {B.}} ~Conversely,  if $\Omega$ satisfies {\rm (1,4)}, {\rm (1.7)} and either   {\rm (a)} or 
{\rm (b)},  then there exists a domain $\tilde\Omega$ with $\partial \tilde\Omega$ is uniformly rectifiable  such that {\rm (1.9)} holds. 

\end{theorem}
\vv

The proof of Part A of Theorem 1.1  is an easy application via Whitney cubes of the theorem  of [HMM2]. The proof of the converse Part B involves a variation on a corona decomposition in [GMT] and  occupies most of this paper.

\vv

\begin{theorem}\label{thm2} 
If $\Omega$ is a  domain satisfying  {\rm (1.4)} and  {\rm (1.7)}, there is $\ve_0 >0$, depending only on the constants in {\rm (1.4)} and {\rm (1.7)}, such that:    

\vv
\noi {\bf {A.}} If {\rm(a)} or {\rm(b)} holds for $\Omega$ then for every 
$0 < \ve < \ve_0$ there is $C(\ve)$ such that if  $p_j \in \Omega \cap B(x,R),$  $x \in \partial \Omega,$ and  $E_j \subset \partial \Omega$ satisfy

\begin{equation}\label{1.10}
\omega(p_j,E_j,\Omega) \geq 1 -\ve,
\end{equation}

\noi and 

\begin{equation}\label{1.11}
 E_j \cap E_k = \emptyset ~~{\rm {if}}~~ k \neq j,
\end{equation}

\noi then 

\begin{equation}\label{1.12}
\sum (\dist(p_j,\partial \Omega)^d \leq C(\ve) R^d.
\end{equation}

\vv
\noi {\bf {B.}} Conversely,  if  for some $0 < \ve < \ve_0$, {\rm (1.10)} 
 and {\rm (1.11)} imply {\rm (1.12)} whenever such $\{p_j\}$ and $\{E_j\}$ exist, then {\rm (a)} and {\rm (b)} hold for $\Omega.$ 

\end{theorem}

\vv

The proof of Part A of Theorem 1.2  is  in Section 4.  It uses a construction  from the beginning of [GMT] and some elementary properties of harmonic measure.  The proof of the converse Part B is deeper.  It runs parallel to the proof of Part B  of Theorem 1.1.

\vv

To illustrate Theorem 1.1 and Theorem 1.2  we consider  Cantor sets.  Let $0 < \lambda < 1/2$ and in  $\R^2$
set $K_{\lambda} = \bigcap_{n \geq 0} K_{\lambda,n} $ where $K_{\lambda,0} = [0,1] \times ]0,1],$ $K_{\lambda,n+1} \subset K_{\lambda,n},$ and $K_{\lambda,n+1}$ is the union of $4^{n+1}$ pairwise disjoint closed squares of side $\lambda^{n+1}$, each  containing one corner of a component square of $K_{\lambda,n}.$    
Then (1.4) and (1.7) hold for 
$\Omega_{\lambda} = \R^2 \setminus  K_{\lambda}.$   Theorem 1.1 
implies  ${\rm (a)}$ or ${\rm (b)}$ holds for $\Omega_{\lambda}$ if and only if $\lambda < 1/4,$ but this can be proved without the harder proof of Theorem 1.1.   If $\lambda \geq 1/4$, $\HH^1$ and harmonic measure for $\C \setminus K_{\lambda}$ are mutually singular ([Ca2], [GM]) and then the easier half  of the proof of Theorem 1.2 in Section 4 shows ${\rm (a)}$ and ${\rm (b)}$ fail.  The case  
$\lambda < 1/4$ is  easier yet because then if $u$ is harmonic on $\Omega_{\lambda}$

$$\int_{B(x,R) \setminus K_{\lambda}} |\nabla u| dy \leq ||u||_{L^{\infty}(\Omega)} \int_{B(x,R) \setminus K_{\lambda}}
 \frac{dy}{\dist(y,K_{\lambda})}\leq CR ||u||_{L^{\infty}(\Omega)}.$$

\vv

\noi When $\lambda < 1/4$ the domain $\widetilde\Omega_{\lambda}$ can be obtained by removing from
$\Omega_{\lambda}$ a continuum of diameter $c\lambda^n$  near the center of each
 $K_{\lambda,n}$, and the converse proof of Theorem 1.1 amounts to constructing similar continuua  in the general case.  There, it is helpful to recall that for  $\lambda < 1/4,$ the harmonic measures for $\Omega_{\lambda}$ and $\widetilde \Omega_{\lambda}$ are
mutually singular.  

\vv  

The Part B converses of Theorem 1.1 and Theorem 1.2 are both corollaries of
Theorem 1.4, which asserts that under   (1.4) and (1.7)  {\rm (a)} and {\rm (b)} are both  equivalent to the existence of a particular corona decomposition on 
$\partial \Omega$ made by comparing harmonic measures to diameters. This corona decomposition is   
similar to the  decomposition in [GMT],  which in the Ahlfors regular case is proved in Proposition 3.1 and Proposition 5.1 of [GMT] to be equivalent to the uniform rectifiability of $\partial \Omega$,  and thus also equivalent to (a) or (b).    
The corona decomposition in [GMT] uses a family of subsets of $\partial \Omega$, often called Christ-David cubes,  
that has been defined only  when $\partial \Omega$ is Ahlfors regular and to make the decomposition more  generally one must first 
define  
  a new family of ``cubes"   
in  $\partial \Omega$. In Proposition 1.3 these cubes are built by repeating the original construction of David [Da] assuming $\Omega$ 
satisfies the condition of Theorem 1.2 but not assuming $\partial \Omega$ is Ahlfors regular.  
The main
difference  between the  corona decomposition in Theorem 1.4  and the one in [GMT]  is this definition of cubes. and

\vv
To state Theorem 1.4 we first explain its setting, to be  discussed  more  in Section 6.

\vv

\begin{proposition}\label{prop1} Assume $\Omega$ is a bounded corkscrew domain satisfying 
{\rm (1.7)} and the conclusion of Theorem {\rm {1.2}} that {\rm (1.10)}  and {\rm (1.11)} imply
{\rm (1.12)}. Then there exists a positive integer $N$ and a family  
$$\SSS = \bigcup_{j \geq 0} {\SSS}_j$$
of Borel subsets of $\partial 
\Omega$ which has properties {\rm (1.13), (1.14), (1.15), (1.16),} {\rm (1.17)} and the ``small boundary property'' {\rm (1.18)}.

\begin{equation}\label{(1.13)}
\diam{S} \sim 2^{-Nj}  ~~~{\rm {if}}~~~ S \in \SSS_j;
\end{equation}

\vv

\begin{equation}\label{(1.14)}
\partial \Omega = \bigcup_{{\SSS}_j}S, ~~~{\rm {for ~all}}~~~ j;
\end{equation}

\vv

\begin{equation}\label{(1.15)}
S \cap  S'  = \emptyset ~~~ {\rm {if}} ~~~ S,~ S'\in {\SSS}_j ~~~{\rm {and}}~~~ S' \neq S;
\end{equation}

\vv

\begin{equation}\label{(1.16)}
 {\rm {if}}~~ j < k, ~~~ S_j \in \SSS_j ~~{\rm {and}} ~S_k \in \SSS_k, ~~~ {\rm {then}} ~~~  S_k \subset S_j ~~~{\rm {or}}~~~ 
S_k \cap S_j = \emptyset.
\end{equation}

\vv

\noi There exists constant $c_0$ such that for all $S \in \SSS$ there exists $x_S \in S$ with  

\begin{equation}\label{(1.17)}
B(x_S,c_0\ell(S)) \cap \partial \Omega \subset S. 
\end{equation}

\vv
\noi   For $0 < \tau < 1$ and $S_j \in \SSS_j$ define

\vv

$$\Delta_{\tau}(S_j) =  \Bigl\{y \in S_j: \dist(y, \partial \Omega \setminus S_j) < \tau 2^{-Nj}\Bigr\} \cup 
\Bigl\{y \in \partial \Omega \setminus S_j: \dist (y,S_j) < \tau 2^{-Nj}\Bigr\},$$ 

\vv

\noi let 

$$\GZ(\tau 2^{-Nj}) = \Bigl\{K =\bigcap_{1 \leq i \leq d+1} \{k_i \tau 2^{-Nj} \leq x_i \leq (k_i + 1) \tau 2^{-Nj}\}, k_i \in \Z \Bigr\}$$

\noi denote the set of closed dyadic cubes in $\R^{d+1}$ of side $2^{-Nj},$ scaled down by $\tau,$ 
 and define 

$$N_{\tau}(S_j) = \#\bigl\{K \in \GZ(\tau 2^{-Nj}): K \cap \Delta_{\tau}(Q) \neq \emptyset\bigr\}.$$

\noi Then there exists a constant $C = C_{sb}$ so that  

\noi 
\begin{equation}\label{(1.18)}
N_{\tau}(S_j) \leq C {\tau}^{(1/C) -d}
\end{equation}

\noi for all $\tau$ and all $S \in \SSS.$ 

\end{proposition}

\vv

Assuming Proposition 1.3 we make the following construction:   By (1.17), (1.13), and (1.4), to each $S \in \SSS$ there corresponds  a ``corkscrew ball" $B(p, \alpha c_0 \ell(S)) \subset \Omega$ with $\dist(p,S) \leq c_0\ell(S).$
Moreover, by (1.7) and  Lemma 3.1 and Lemma 3.2  from Section 3 below, there exist for any
$0 < \ve < 1/2$  constants

\begin{equation}\label{(1.19)}
 2^{-N-1}c_0 < c_3 < 4c_3  < c_0
\end{equation}

\noi  depending  on $\ve$,  the constants in (1.4)  and (1.7), and the constants $c_1, c_2$ from Section 3, but not on $N,$  
such that for every $S \in \SSS$ there exists  a ball 
$B_S = B(p_S,c_3\ell(S))$ satisfying: 

\begin{equation}\label{1.20)}
B_S = B(p_S,c_3\ell(S)) \subset 4B_S = B(p_S,4c_3\ell(S)) \subset \Omega \cap B(x_S,\frac{c_0}{2}\ell(S));
\end{equation}

\noindent and

\begin{equation}\label{(1.21)}
\inf_{p \in 2B_S}\Bigl\{\omega\bigl(p,S \cap B(x_S,c_0\ell(S)), \Omega \cap B(x_S,c_0\ell(S))\bigr)\Bigr\} \geq 1-\ve.
\end{equation}

\vv
\noi We can also take $N$ so large that if $S \cap S' = \emptyset$

\begin{equation}\label{(1.22)}
B_S \cap B_{S'} = \emptyset 
\end{equation}

\noi and if $\ell(S') > \ell(S)$

\begin{equation}\label{(1.23)}
2B_{S'} \cap B(x_S, c_0\ell(S)) =\emptyset.
\end{equation}

\vv

\noi Indeed, when $N$ is large  (1.22) follows from  (1.13) and 
(1.16) if $\ell(S) \neq \ell(S')$ , and if $\ell(S) = \ell(S')$ 
(1.15) and (3.4) imply 
(1.22) since $\ve < 1/2.$  If $\ell(S') > \ell(S)$ then (1.23)  holds by  (1.19).

\vv

For $S\in \SSS$ and $\lambda >1$ define
$\lambda S = \{x: \dist(x,S) \leq (\lambda -1)\ell(S)\}.$ 
Let $0 < \delta \lesssim 1$ and $A  \gtrsim 1$ be fixed constants.  For  $S_0 \in \SSS$ and $S \in \SSS$ with $S \subset S_0$,   we say $S \in \HD(S_0)$
(for ``high density") if  $S$ is a maximal cube for which

\begin{equation}\label{(1.24)}
\inf_{p \in B_{S_0}}\omega(p,2S) \geq A \Bigl(\frac{\ell(S)}{\ell(S_0)}\Bigr)^d,  
\end{equation}

\noi and we say $S \in \LD(S_0)$ (for ``low density) if $S$ is maximal for

\begin{equation}\label{(1.25)}
\sup_{p \in B_{S_0}}\omega(p,S) \leq \delta \Bigl(\frac{\ell(S)}{\ell(S_0)}\Bigr)^d.  
\end{equation}

\vv
\noindent
By (3.2) and Harnack's inequality   

\begin{equation}\label{(1.26)}
\sup_{p \in B_{S_0}}\omega(p,S) \leq c_5\inf_{q \in B_{S_0}}\omega(q,2S)
\end{equation}

\noi for some constant $c_5$ and we can assume $A > c_5\delta$ so that $\HD(S_0) \cap \LD(S_0) = \emptyset.$

\vv 
For  each $S_0 \in \SSS$ let 

\begin{equation}\label{(1.27)}
G_1(S_0) =\bigl\{S \in \LD(S_0) \cup \HD(S_0): S {\rm ~{ is ~ maximal}}\bigr\}.
\end {equation}

\vv

\noindent We call $G_1(S_0)$ the {\it {first generation of descendants}} of $S_0$,  and we define later generations  inductively:

\begin{equation}\label{(1.28)}
G_k(S_0) = \bigcup_{S \in G_{k-1}(S_0)} G_1(S).
\end {equation}

\vv

Proposition 1.3 will be proved in Section 5 after Part A of Theorem 1.2 has  been  proved in Section 4.   Therefore it is not inconsistent assume the conclusions of Proposition 1.3 after assuming (a) or (b) in Part A  of Theorem 1.4.

\vv     

\begin{theorem}\label{thm3} If $\Omega$ is a  domain satisfying  {\rm (1.4)} and  {\rm (1.7)} there is $\ve_1 >0$, depending only on the constants in {\rm (1.4)} and {\rm (1.7)}, such that: 

\vv
\noi {\bf {A.}} ~ Assume     
 {\rm (a)} or {\rm (b)} holds for $\Omega$ and let $\SSS$ be a family of subsets of $\partial \Omega$ satisfying Proposition {\rm (1.3)}.  Then there exists  $A_0$ such that whenever $0 < \ve < \ve_1$, $0 < \delta < \frac{\ve}{3}$, and 
 $A > {\rm {Max}}(A_0,c_5\delta)$, 
 there exists a constant $C = C(\ve,\delta,d,A)$ such that  for any $S_0 \in \SSS$ 

\begin{equation}\label{(1.29)}
\sum_{k=1}^{\infty} \sum_{G_k(S_0)} \ell(S)^d  \leq C\ell(S_0)^d.
\end {equation}

\vv

\noi {\bf {B.}}~ Conversely,  assume there exists a family $\SSS$ of subsets of $\partial \Omega$  satisfying Proposition {\rm (1.3)} and {\rm {(1.20), (1.21)}} and {\rm (1.22)}, assume {\rm {(1.24), (1.25), (1.26)}}  and 
{\rm {(1.27)}} hold for  some  $\ve, \delta$ and $A$ with $0 < \ve < \ve_1$, 
$0 < \delta < \frac{\ve}{3}$, and $A > c_5\delta,$
and further  assume

\vv

{\rm {(i)}} $\SSS$ satisfies {\rm (1.29)}, and

\vv

{\rm {(ii)}} there exists $C > 0$ such that if $B$ is a ball, 
$\{S_j\} 
\subset \SSS,$  
$\bigcup S_j \subset B$   
and
$S_j \cap S_k  = \emptyset$ for  $j \neq k$, then $\sum \ell(S_j)^d \leq C \diam (B)^d.$

\vv

\noi Then {\rm (a)} and {\rm (b)} hold for $\Omega.$ 
\end{theorem}

\vv

Part A  of Theorem 1.1 is proved in Section 2, without assuming (1.4) or (1.7).  In Section 3 we give three lemmas from [An] and [GMT] which lead to the proof in Section 4 of Part A  of Theorem 1.2.     

The proofs of  Theorem 1.1 Part B,  Theorem 1.2 Part B, and Theorem 1.4 are convoluted. In Section 5 the conclusion of Theorem 1.2 is used to define the cube family $\SSS$ and prove Proposition 1.3.    In  Section 6    
properties of $\SSS$ and 
 the construction from [GMT] yield a proof  Part A  of Theorem 1.4 (and  thereby extend Proposition 3.1 of [GMT] to domains satisfying (1.4) and (1.7)).    Then in Section 7  Part A of Theorem 1.4 is used to construct a subdomain $\tilde \Omega \subset \Omega$
such that $\partial \Omega \subset \partial \tilde \Omega$ and  $\partial \tilde \Omega$ is Ahlfors regular and in Section 8  the crucial   
generation sum (1.29) for $\Omega$ is used to control  the corresponding  sum for  $\tilde\Omega$. With Lemma 6.2  and  Proposition 5.1 of [GMT], that implies  $\partial \tilde \Omega$ is uniformly rectifiable, and therefore proves  Part B of 
Theorem 1.1.  Finally,  the proof of Theorem 1.4 Part B    follows from Theorem 1.1 Part A and the proof of  Theorem 1.1 Part B, and  the  proof of Theorem 1.2 Part B is a word-for-word repeat of the same argument.   An outline of the logic is:


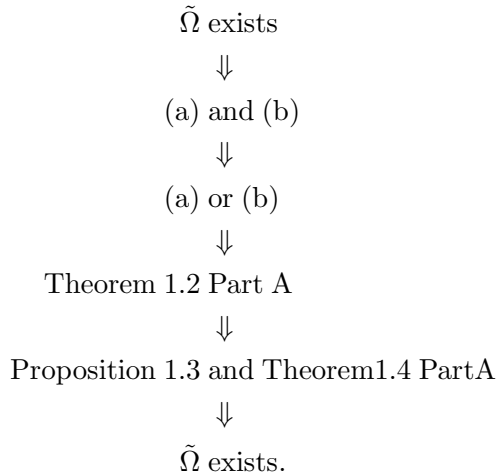
\begin{align}
\tilde\Omega ~~ &{\rm {exists}} \nonumber \\
                   &\Downarrow \nonumber \\
{\rm (a)} ~~ &{\rm {and}} ~~ {\rm (b)} \nonumber \\
             &\Downarrow \nonumber \\
{\rm (a)} ~~ &{\rm {or}} ~~ {\rm (b)} \nonumber  \\
                 &\Downarrow \nonumber \\
{\rm {Theorem  ~ 1.2}} ~~~ &{\rm {Part ~ A}}\nonumber \\
                     &\Downarrow \nonumber \\
{\rm {Proposition ~~1.3}} ~~&{\rm {and}} ~~{\rm {Theorem 1.4 ~~ Part A}} \nonumber \\
                      &\Downarrow \nonumber \\
\tilde\Omega ~~ &{\rm {exists}}.  
 \nonumber
\end{align}

\vv

 A reading of the proofs will show that $\ve-$ approximation of all harmonic functions with
$\sup_{\Omega} |u|  \leq 1$ for some  fixed small $\ve$ is equivalent to the other conclusions of all three theorems.

\vv
The argument in this paper entails many constants.  Constants $C$ or $C_j$ are large and may vary from use to use, 
but the constants $c_0, c_1, ......$ are small and sometimes interdependent. They are written so that $c_j$ can depend on $c_k$ only if $k < j.$

\vv

\section{Proof of  Theorem 1.1 Part A.}

We recall the Whitney decomposition of $\Omega$ into  cubes $\Omega = \bigcup_{\cW}Q$.  Each $Q \in \cW = \cW(\Omega)$ is a closed dyadic cube,

\begin{equation}\label{2.1}
Q = \bigcap_{1 \leq j \leq d+1} \bigl\{k_j2^{-n} \leq x_j \leq (k_j+1)2^{-n}
\bigr\}
\end{equation}

\noi with  $n$ and $k_j$ integers.
 If  $Q_1, Q_2 \in \cW$ then  

\begin{equation}\label{2.2}
Q_1 \subset  Q_2, ~~~ Q_2 \subset Q_1, ~~~ {\rm {or}} ~~~ Q_1^o \cap Q_2^o = \emptyset,  
\end{equation}

\vv
\noi where $Q^o$ denotes the interior of $Q$.  There are constants $1 < c_6 < c_7 < 3$ such that for all $Q \in \cW$
\begin{equation}\label{2.3}
c_6 Q \cap \partial \Omega = \emptyset  ~~~{\rm {but}}~~~ c_7 Q \cap \partial \Omega \neq \emptyset,
\end{equation}
\vv
\noi  where $\ell(Q)$ is the sidelength of $Q$ and $cQ$ is the concentric closed cube having sidelength $c\ell(Q).$ 

\vv
Assume $\Omega$ and $\widetilde \Omega$ satisfy  condition (1.9) from Theorem 1.1, let $u$ be an harmonic function on $\Omega$ with
 $\sup_{\Omega} |u(y)| \leq 1$, and let $Q \in \cW(\Omega).$  We fix a constant $1 < c_8 < c_6$ and consider two cases.

\vv
\noi {\bf {Case I:}} 
$c_{8} Q \cap \partial \widetilde \Omega = \emptyset.$

\vv
 In this case there is $C_1 = C_1(d, c_7, c_8)$ such that 
$\dist(y,\partial \Omega) \leq C_1 \dist(y, \partial \widetilde \Omega)$
for all $y \in Q$, so that 
\vv

\begin{equation}\label{case1}
\int_Q |\nabla u(y)|^2 \dist(y,\partial \Omega) dy \leq 
C_1\int_Q |\nabla u(y)|^2 \dist(y,\partial \widetilde \Omega) dy.
\end{equation}

\vv

\noi{\bf {Case II:}} 
$c_{8} Q \cap  \partial \widetilde \Omega \neq \emptyset.$

\vv

In this case  Harnack's inequality  gives  
$\sup_Q|\nabla u(y)| \leq \frac{C_2}{\ell(Q)},$
for $C_2 = C_2(d,c_7)$,  so that

\begin{equation}\label{case2}
\int_Q |\nabla u(y)|^2 \dist(y,\partial \Omega) dy
\leq C_2^2 (1 + c_8)^\frac{d+1}{2} {\ell(Q)}^d = C_3 {\ell(Q)}^d.
\end{equation}

\vv

\noi Now consider a ball $B = B(x,r)$, with   $x \in \partial \Omega,  r < \diam \Omega$,  and let 
 $$\cW_B = \{Q \in \cW(\Omega): Q \cap B \neq \emptyset\},$$

\vv
\noi and for $J = I$ or $II$ 
let $\cW_{B,J}$ be the set of Case J cubes in $\cW_B.$
Also note that by (2.3) 

\begin{equation}\label{bigB}
\bigcup_{\cW_B} c_6Q \subset B(x, C_4r)
\end{equation}

\noi for a constant  $C_4$ depending only $c_6$ and $c_7.$ 
Since $\HH^{d+1}(\partial \widetilde \Omega \setminus \partial \Omega) =0$ (because $\partial \widetilde \Omega$ 
is Ahlfors regular) we have

$$\int_B |\nabla u(y)|^2 \dist(y,\partial \Omega) dy \leq 
\sum_{\cW_B} \int_Q |\nabla u(y)|^2 \dist(y,\partial \Omega) dy =
\sum_{\cW_{B,I}} + \sum_{\cW_{B,II}}.$$

\vv
\noindent To estimate $\sum_{\cW_{B,I}}$ we use  (2.4), (2.6), the uniform rectifiability of $\partial \widetilde \Omega,$ and the theorem of [HMM2] to get  

\begin{equation}\label{case11}
\sum_{\cW_{B,I}} \leq 
C_1 \int_{B(x, C_4r)} |\nabla u(y)|^2 \dist(y, \partial \widetilde \Omega) dy \leq C(C_4r)^d.
\end{equation} 

\noindent For estimating  $\sum_{\cW_{B,II}}$  the only available inequality is  
  
$$\sum_{\cW_{B,II}} \leq C_3 \sum_{\cW_{B,II}}  {\ell(Q)}^d$$

\noi from (2.5).  But in Case II

\begin{equation}\label{(2.8)}
{\ell(Q)}^d \leq C_5 \HH^d(c_6Q \cap \partial \widetilde \Omega) 
\end{equation}

\noi because $\partial \widetilde \Omega$ is Ahlfors regular
\noi and by  (2.2) and (2.3)
no point lies in more than $N = N(c_6, c_7, d)$ cubes $c_6Q, ~ Q \in \cW.$
Therefore (2.5), (2.6), and the Ahlfors  regularity of $\partial \widetilde \Omega$ imply

\begin{equation}\label{case21}
\sum_{\cW_{B,II}} \leq   
C_5\sum_{\cW_{B,II}} {\ell(Q)}^d \leq C_5N \HH^d(B(x,C_4r)) \leq C_5N (C_4r)^d. 
\end{equation}

\noi Thus by (2.7), (2.5) and (2.9), (a) holds for all bounded harmonic $u$.

\vv

To prove (b) let $u$ be an harmonic function on $\Omega$, let $\ve > 0$ and consider the Case I and Case II cubes in $\cW(\Omega)$.   Write 

$$U = \bigcup_{\rm {Case ~ II}} Q,~~ V = \bigcup_{\rm {Case ~ I}} Q,$$

\noi and 

$$\Gamma = \Omega \cap \partial V  = \Omega \cap \partial U.$$

\vv

\noi Let $g \in W^{1,1}(\tilde \Omega)$ satisfy (1.2) and (1.3) for $u$ on $\tilde \Omega$ 
and define
$G = g \chi_U + u \chi_{V \cup \Gamma}.$ Then $||u -G||_{L^{\infty}(\Omega)} < \ve,$ and on $\Omega$ the distribution 

$$\nabla G = \chi_U \nabla g + \chi_V \nabla u + \nu$$ 

\noi where $\nu$ is an $\R^{d+1}$-valued measure that accounts for the jump between  $g$ and $u$ across
$\Gamma$  and has total variation 
$|\nu| \leq \ve \chi_{\Gamma} \HH^d.$
 Let $x \in \partial \Omega$ and $r >0.$ Then by the proof of (a)

$$\int_{B(x,r) \cap (U \cup V)}|\nabla G| dy \leq C r^d,$$

\noi and because $\partial \tilde \Omega$ is Ahlfors regular,  (2.8) implies 

$$|\nu|(B(x,r) \cap \Omega) \leq C\ve r^d.$$

\noi Hence (1.3) holds for the vector measure $\nabla G.$ 

\vv

To replace $G$ by a $W^{1,1}_{\rm {loc}}$ function let $\eta > 0$ be small, write
$$\psi_{\eta}(y) = \eta^{-(d+1)} \psi\bigl(\frac{y}{\eta}\bigr)$$ where 
$\psi \in C^{\infty}(R^{d+1})$ is
a non-negative radial function, compactly supported in $B(0,1),$ 
and satisfying $\int_{R^{d+1}}  \psi dy =1,$ and for $y \in \Omega$ define by convolution

$$\tilde G(y) = G * \psi_{\eta \dist(y,\partial \Omega)}(y).$$

\noi Then $\tilde G \in W^{1,1}_{\rm {loc}}(\Omega)$ and (1.2) and (1.3) hold for $\tilde G$
 and $u$.

\vv
\section{Three Lemmas.}

 Recall we assume
(1.7)  so that the harmonic measure  $\omega(p,E) = \omega(p,E, \Omega)$ exists for $p \in \Omega$ and Borel $E \subset \partial \Omega.$   The first lemma is Lemma 3 from [An].

\begin{lemma}\label{lem2}  $\Omega$ satisfies {\rm (1.7)} with constant $\beta$ if and only if there exists 
$\eta = \eta(\beta) <1$ such that for all $x \in \partial \Omega$ and all $r >0$  

\begin{equation}\label{(3.1)}
\sup_{B(x,r) \cap \Omega}\omega(p,\partial B(x,2r) \setminus \Omega, \Omega \cap B(x,2r)) \leq \eta.
\end{equation}

\end{lemma}

\vv

The second lemma is a well-known consequence  of Lemma 3.1 and induction.

\vv

\begin{lemma}\label{lem3}  Assume $\Omega$ satisfies {\rm (1.4)} and {\rm (1.7)} and 
let $0 < \ve < \frac{1}{2}.$
There are   constants $c_1$  and $c_2$  depending only on $\ve$ and the constants $\alpha$ and $\beta$ in
{\rm (1.4)} and {\rm (1.7)}, such that whenever $x \in \partial \Omega$ and $r < \diam \Omega,$ there exists a ball 
$B = B(p,c_1r)$ such that 

\begin{equation}\label{(3.1)}
4B = B(p,4c_1r) \subset \Omega \cap B(x,r),
\end{equation}

\vv

\begin{equation}\label{(3.2)}
\dist(p,\partial \Omega) < c_2r,
\end{equation}

\noindent and

\begin{equation}\label{(3.3)}
\inf_{q \in 2B}\omega(q,\partial \Omega \cap B(x,r),\Omega \cap B(x,r)) > 1-\ve.
\end{equation}

\end{lemma}

\vv

\begin{proof}

By  the maximum principle and induction (3.1)  implies

\begin{equation}\label{hmub2}
\sup_{B(x,r) \cap \Omega}\omega(p,\partial B(x,2^N r) \setminus \Omega, \Omega \cap  B(x,2^N r)) < \eta^N.
\end{equation}

\vv

\noi For $\ve >0$ take $N$ with $\eta^N < \ve$ and set $C_1 = 1 + 2^N.$  For any $p \in \Omega$ take $x \in \partial \Omega$ such that $|x-p| = \dist(p, \partial \Omega)$.  Applying (3.5) with $r = |x-p|$ yields
 
\begin{equation}\label{hmub3}
\omega(p, \partial \Omega \setminus B(p, C_1\dist(p,\partial \Omega)),\Omega) < \ve.
\end{equation}

\vv

\noi  By (1.4), $\Omega \cap B(x, \frac{r}{1 + C_1})$ contains a ball
 $B = B(p, \frac{\alpha r}{1 +C_1}).$ Therefore   (3.2) holds with

$$c_1 = \frac{\alpha}{4(1 +C_1)}$$ 

\noi and (3.3) holds with 

$$c_2 = 
\frac{1}{1 +C_1}.$$

\vv
\noi  If $q \in 2B = B(p,\frac{{\alpha}r}{2(1 +C_1)})$ then by (3.2) 
$\dist(q,\partial \Omega) \leq |q -x| \leq \frac{r}{1 + C_1}$.  Therefore 
$B(q,C_1\dist(q,\partial\Omega)) \subset  B(x,r)$, so that (3.6)  implies (3.4). 

\end{proof}

\vv

The next lemma is similar to Lemma 3.3 of [GMT].

\vv
\begin{lemma}\label{lem4} Assume $\Omega$ satisfies {\rm (1.4)} and {\rm (1.7)}.  Then there exists $\ve_0 >0$ and constants $c_{9}$ and $c_{10}$ depending only on $d$  and the constants $\alpha$ and $\beta$ of {\rm (1.4)} and {\rm (1.7)} such that if $0 < \ve < \ve_0$ and
\begin{itemize}
\vv
\item[(i)] $S\subset \partial \Omega$ is a Borel set, $x \in S$, $0 < r < \diam(\Omega),$  and $B(x,r) \cap \partial \Omega \subset \ S$, 
\vv
\item[(ii)] the ball $B_S = B(p_S,c_1r)$ satisfies {\rm (3.2), (3.3)} and {\rm (3.4)}  from Lemma {\rm 3.2}, 
\vv
\item[(iii)] $E_S \subset S \cap B(x,r)$ is a compact set such that 
\begin{equation}\label{thinout}
\inf_{2B_S} \omega(q,E_S,\Omega) \geq 1 - \ve,
\end{equation} 
\end{itemize}
\vv
\noi then there exists a non-negative harmonic function 
$u_S$  on $\Omega$ and a Borel function  $f_S$ such that  
$$0 \leq f_S \leq \chi_{E_S}$$ and for all $p \in \Omega,$

\begin{equation}\label{repf}
u_S(p) = \int_{E_S} f_S(y) d\omega(p,y,\Omega),
\end{equation}

\vv

\begin{equation}\label{ulb}
\inf_{B_S}u_S(p) \geq c_{9},
\end{equation}

\noindent and there exists a unit vector $\vec{e_Q} \in \R^{d+1}$ such that

\begin{equation}\label{gradlb}
\inf_{B_S} |\nabla u_Q(p) \cdot \vec{e_S}| \geq \frac{c_{10}}{c_1r}.
\end{equation}

\end{lemma}

\vv
\noi The  right side of (3.10) is  so written to display the radius $c_1r$ of $B_S.$  

\vv
\begin{proof}  Take $q_S \in S \cap \partial \Omega$ with $|q_S -p_S| < 2\dist(p_S,\partial \Omega).$ 
By (3.2) and (3.3)  we have 

\begin{equation}\label{relpos}
4c_1r < |p_S -q_S| < 2c_2r.
\end{equation}

\vv

\noi {\bf {Case I.}} $d \geq 2.$  By (1.7) and the 
definition of capacity there exists a positive measure $\mu_S$ supported on $\overline B(q_S,c_1 r) \cap \partial \Omega$ with  
$\int d\mu_S > \beta (c_1 r)^{d-1}$ such that the potential

$$U_S(p) = \int |p - y|^{1-d} d\mu_S(y)$$

\noi is harmonic on $\R^{d+1} \setminus {\rm {supp}} \mu_S  \supset \Omega$,  and satisfies

\begin{equation}\label{potbound}
0 < U_{S}(p) \leq 1 
\end{equation}

\noi for all $p \in \R^{d+1}.$  By Egoroff's theorem there is a compact set $F_S \subset \overline B(q_0,c_1 r) \cap \partial \Omega$ 
such that $\mu_S(F_S) \geq \beta (c_1 r)^{d-1}$ and 

$$\int_{B(p,\eta)} |p-y|^{1-d} d\mu_S(y) \to 0 ~ (\eta \to 0)$$

\noi  uniformly
on $F_S.$ Redefine $U_S$ to be 

\begin{equation}\label{pot}
U_S(p) = \int_{F_S} |p - y|^{1-d} d\mu_S(y).
\end{equation}

\noi Then $U_S$ is continuous on $\R^{d+1}$, harmonic on $\R^{d+1} \setminus F_S \supset \Omega$,  and satisfies
(3.12).

\vv

 By (3.11) and  (3.13),

\begin{equation}\label{potlb}
\inf_{2B_S} U_S(p) \geq \beta \Biggl(\frac{c_1 r}{|p_S -q_S| + 3c_1r}\Biggr)^{d-1} = \beta 7^{1-d} =c'_{9}
\end {equation}

\vv
 
Let $\vec e_S = \frac{\overrightarrow{(q_S -p_S)}}{|q_S -p_S|}.$ Then by (3.11) we have

\begin{equation}\label{relpos1}
\inf \Bigl\{\vec e_S \cdot \frac{\overrightarrow{(q-p)}}{|q-p|} q \in F_S, p \in B_S \Bigr\}
= \frac{c_2}{c_1} = 
\frac{4}{\alpha}. 
\end{equation}

\noi Hence by  (3.11), (3.13),  (3.15), and the formula 

\begin{equation}\label{grad}
\nabla U_S(p) = (1-d) \int_{F_S} \frac{\overrightarrow {(p-y)}}{|p-y|^{d+1}}d\mu_S(y),
\end{equation}

\noi we have  on $B_S$

\begin{equation}\label{biggrad}
|\nabla U_S(p) \cdot \vec e_S|  \geq \frac{4}{\alpha} \frac{(d-1)\beta {c_1r}^{d-1}}{(2c_1r + 2c_2r)^d}= 
\frac{c'_{10}}{c_1r}, 
\end{equation}

\noi in which 

$$c'_{10} = \frac{d-1}{{2{c_1}}^{d-2}} \frac{\beta}{\alpha} \bigl(\frac{\alpha}{4+\alpha}\bigr)^d$$

\noindent depends only on $d$, $\alpha$ and $\beta.$  
Since $U_S$ is continuous on $\overline \Omega$, 

$$U_{S}(p) = \int_{\partial \Omega} g_S(y) d\omega(p,y,\Omega)$$

\noindent  with continuous $g_S = U_{S}\vert \partial \Omega.$ 
 Set $f_S = \chi_{E_S}g_S$ and define $u_S$ by (3.8).  Finally take 

$$\ve_0 < {\rm {Min}} \Bigl(\frac{c'_{9}}{2}, \frac{c'_{10}}{3}\Bigr),$$

\noi assume $0 < \ve < \ve_0$, and assume Lemma 3.2 holds for
 $c_1$, $c_2$, $\alpha$ and $\ve.$  Then since  $|f_Q| \leq 1$, (3.7)  yields $\sup_{2B_S}|U_S - u_S| \leq \ve$.  Hence  (3,14) implies (4.4) for $c_{9} = \frac{c'_{9}}{2}$, and by (3.7) and Harnack's inequality 
$\sup_{B_S} |\nabla (U_S - u_S)| \leq \frac{2\ve}{c_1 r}$.  so that  (3.7) implies  (4.5) for 
$c_{10} =  \frac{c_{10}'}{3}.$

\vv
\vv

\noi {\bf {Case II:}}  $d=1$.  Decreasing $c_1$ and $c_2$ if necessary,  we have, again by Egoroff's theorem,  compact sets 
$F^{\pm}_S 
\subset \overline B(x,r) \cap S$ such that ${\rm {Cap}}(F^{\pm}_S) \geq \frac{\beta c_1r}{2} \equiv  e^{-\gamma}$ and  probability measures $\mu_{\pm}$ supported on $F^{\pm}_S$ so that the logarithmic potentials

$$U_{\pm}(p) = \int_{F_S^{\pm}} \log\frac{1}{|p-y|} d\mu_{\pm}(y)$$
are continuous on $\R^2$ and harmonic on $\R^2 \setminus F_S^{\pm}$ and satisfy $U_{\pm} < \gamma$ on $\R^2 \setminus F_S^{\pm}$ and for small $\eta,$  $\gamma - \eta \leq U_{\pm} \leq \gamma$ on $F_S^{\pm}.$  Because capacity is bounded by diameter, we can, by choices of $c_1$ and $c_2$,  position 
$F_S^{\pm}$ so that 

$$F_S^{+} \subset B(p_S,2c_2r)$$

\noi but

$$F_S^{-} \subset \R^2 \setminus B(p_S,4c_2r).$$ 

\vv

\noi Then on $\R^2 \setminus (F_Q^{+} \cup F_S^{-})$ the  function $U^{+} - U^{-}$ is harmonic and bounded, because the logarithmic singularities at $\infty$ cancel, and by the
 choices of $F_S^{\pm}$, 

$$\sup_{F^{+} \cup F^{-}}|U^{+} - U^{-}| \leq \gamma - \log\Bigl(\frac{1}{2c_r2}\Bigr) = 
\log \Bigl(\frac{4c_2}{\beta c_1}\Bigr),$$

\vv

$$\inf_{2B_S} (U^{+} - U^{-}) \geq \log\Bigl(\frac{1}{2c_2r - 2c_1r}\Bigr) - \log\Bigl(\frac{1}{4c_2r + 2c_1r}\Bigr) = \log\Bigl(\frac{2c_2 + c_1}{c_2 - c_1}\Bigr),$$ 

\vv
\noi and for some unit vector $\vec{e_S}$,

$$\inf_{B_S}|\Bigl| \nabla(U^{+} - U^{-}) \cdot \vec{e_S}\Bigr| \geq \frac{c_{10}''}{r}.$$

\noi Then (3.8), (3.9) and (3.10) hold for 

$$f_S = \Bigl(2 \log\bigl(\frac{4c_2}{c_1}\bigr)\Bigr)^{-1} \Bigl(\log\bigl(\frac{4c_2}{c_1}\bigr) + U^{+} - U^{-}\Bigr)\chi_{E_S}.$$

\end{proof}

\vv 

 \section{Proof of Theorem 1.2 Part A.}

We follow the proof of Lemma 3.7 of [GMT].  Replacing $\ve$ by $\frac{\ve}{4}$  and $R$ by $CR$,  $C > 1,$  and setting $r_j = \dist(p_j,\partial \Omega)$ and $B_j = B(p_j, r_j),$ we can by Lemma 3.2 and Harnack's inequality assume $E_j \subset B(x,R)$,  $4B_j = B(p_j,4r_j) \subset \Omega \cap B(x,R)$ and

$$\inf_{2B_j} \omega(p,E_j,\Omega) > 1 - \frac{\ve}{2}. \eqno (4.1)$$

\noi  Then the conclusion of Theorem 1.2 Part A is immediate from:

\vv
\begin{lemma} Assume {\rm (1.4)}, {\rm (1.7)} and either {\rm (a)} or {\rm (b)} hold for $\Omega.$ Then if $0 < \ve < \ve_0$ there is $C(\ve)$ such that if for $j = 1, 2, ....$ 
there exist balls  
$B_j = B(p_j,r_j) \subset \Omega \cap B(x,R)$, $x \in \partial \Omega$
and sets $E_j \subset \partial \Omega$ with {\rm (4.1)}
and

$$E _j \cap E_k = \emptyset, j \neq k, \eqno (4.2)$$

\noi then 

$$\sum r_j^d \leq C(\ve) R^d. \eqno (4.3)$$

\end{lemma}

\vv

\begin{proof} By   
Lemma 3.3 there exists  a Borel function $0 \leq f_j \leq \chi_{E_j},$ 
such that the harmonic function 

\begin{equation}\label{repf1}
u_j(p) = \int_{E_j} f_j(y) d\omega(p,y,\Omega),
\end{equation}

\noindent satisfies 

\begin{equation}\label{ulb}
\inf_{2B_j}u_j(p) \geq c_{12},
\end{equation}

\noindent and there exists a unit vector $\vec{e_j} \in \R^{d+1}$ such that

\begin{equation}\label{gradlb}
\inf_{B_j} \bigl|\nabla u_j(p) \cdot \vec{e_j}\bigr| \geq \frac{c_{12}}{r_j}.
\end{equation}

\noi Set $u = \sum u_j.$ 
Then by (4.1)  we have 
$\sup_{2B_{j}}|u-u_j| \leq \frac{\ve}{2}$, so that by Harnack's inequality
$\sup_{B_{j}}|\nabla (u-u_j)| \leq \frac{2\ve}{r_j}$. Therefore  

$$|\nabla u| > \frac{{c_{11} - 3\ve}}{r_j} $$

\noi on $B_{j}$ and

\begin{equation}\label{contra}
\int_{B_{j} \cap \Omega} |\nabla u(x)|^2 \dist(x,\partial\Omega) dx \geq c_{12} r_j^d.
\end{equation}

\noindent Assuming {\rm (a)} holds on $\Omega$ with constant $C$ and summing, we obtain

$$\sum_j (\dist(p_j,\partial \Omega)^d \leq \frac{1}{c_{12}} \int_{B \cap \partial \Omega} 
 |\nabla u(x)|^2 \dist(x,\partial \Omega) \leq C R^d,$$

\noi which yields (4.3) when {\rm (a)} holds. 

\vv

 Now assume {\rm (b)} holds for $\Omega$ and $\ve < \frac{c_{11}}{3}$. If $g \in W^{1,1}_{\rm {loc}}(\Omega)$ satisfies (1.2) for $u$ and $\ve < \frac{c_{11}}{3}$  
then, using (4.6) and (4.7) for $u_j$, we obtain

$$\int_{B_j} |\nabla g(x)|dx \geq c_{13} r_j^d.$$

\noi Thus from (a) or (b) we conclude that (4.3) holds.  \end{proof}

\vv

We  note two corollaries of Lemma 4.1.

\vv
\begin{corollary}\label{coro1} Let $\Omega \subset \R^{d+1}$ be a corkscrew domain for which 
{\rm (1.7)} holds.  If {\rm (a)} or {\rm (b)} holds for $\Omega$ then there a constant $C > 0$ such that for all $x \in \partial \Omega$ and all $r > 0,$ 

\begin{equation}\label{upperAR}
\HH^d(B(x,r) \cap \partial \Omega) \leq Cr^d.
\end{equation}
\end{corollary} 

\vv
\begin{proof}  Cover any compact $K \subset B(x,R) \cap \partial \Omega$ by a minimal set $\cF$ of $N_n$ distinct closed dyadic cubes of side $2^{-n}$. Partition $\cF$ into $3^{d+1}$ disjoint families 
$\cF'$ so that  $\dist(Q_1,Q_2)  \geq  2^{-n}$ if $Q_1 \neq  Q_2 \in \cF'$ and fix  any such family  
$\cF$.  By (1.4) and (1.7) and Lemma 3.2 there exists $c_{14}$ so that for every 
$Q_j \in \cF'$ there exists a ball 
$B_j = B(p_j, c_{14} 2^{-n}) \subset \Omega \cap {\frac{5}{4}} Q_j$  
with $\inf_{B_j} \omega(p,Q_j \cap \partial \Omega, \Omega) > 1- \ve$,  where $\ve$ fixed and small.
Then by Lemma 4.1   

$$ (c_{14}2^n)^d  \#{\cF'} \leq C(\ve) r^d$$

\noi which yields

$$\HH^d(K) \leq 3^{d+1} c_{14}^{-d} C(\ve) r^d.$$
\end{proof}

Merged with the results of  
[HMM2] and [GMT] Corollary 4.2  yields: 

\vv

\begin{corollary}\label{coro2}  If $\Omega \subset \R^{d+1}$ is a corkscrew domain for which there exists constant $c >0$ such that for all $x \in \partial \Omega$ and all $0 < R < \rm{diam}(\partial \Omega),$

\begin{equation}\label{lowerAR}
\HH^d(B(x,r) \cap \partial \Omega) \geq cr^d,
\end{equation}

\noi then {\rm (a)} or {\rm (b)} holds for $\Omega$ if and only if $\partial \Omega$ is uniformly rectifiable.
\end{corollary}

\vv

\section{Modified Christ-David cubes}  

To  prove Theorem 1.3. we follow  the construction in  [Da] very closely, 
although the arguments from  [Ch] and [DM] would also work.  
To start we use {\rm (a)} or {\rm (b)} to get a grip on the small boundary condition (1.18).

\vv

\begin{lemma}\label{lem6} Let $0 < \eta < 1$  and let $N$  be a positive  integer.  Assume $\Omega$ is a bounded corkscrew domain with {\rm (1.7)} and assume the conclusion of Theorem {\rm {1.2}} holds for
$\Omega.$   Then  for any $x \in \partial \Omega$ and any $j \in N$ there exists an open ball 
$B_j(x) = B_j(x,r)$ having center $x$ and radius

$$r \in (2^{-Nj},(1 + \eta)2^{-Nj})$$  

\noi such that if 

$$\Delta_j(x) = B_j(x) \cap \partial \Omega,$$ 

 \vv

\begin{align}
E_j(x)  &= \bigl\{y \in \Delta_j(x): {\rm {dist}}(y, \partial \Omega \setminus \Delta_j(x)) < \eta^2 2^{-Nj} \bigr\} \nonumber \\
&\cup \{y \in \partial \Omega \setminus \Delta_j(x): {\rm {dist}}(y, \Delta_j(x)) < \eta^2 2^{-Nj})\} \nonumber
\end{align}

\vv
\noi and  $m_j(x)$ is the minimum number of closed balls ${\overline{B(p,\eta^2 2^{-Nj})}}$ needed to
 cover $E_j(x)$, 
then

\begin{equation}\label{rings}
m_j(x) \leq C_d \eta^{1 -2d}, 
\end{equation}

\vv

\noi in which  the constant $C_d$ depends only on $d$ and the constant in {\rm (1.12)}.

\end{lemma}

\vv

\begin{proof}  Partition the closed shell 
$\Sigma = \overline {B(x,(1 +\eta)2^{-Nj})} \setminus B(x,2^{-Nj})$
into a family $\cR$ of at most $1 + [\frac{1}{\eta}]$ closed shells of width $\eta^2 2^{-Nj}$.
 Fix $2^{-n} \sim \eta^2 2^{-Nj}$,  let 
$\cE$ be the set of closed dyadic cubes $Q$ of side $2^{-n}$ such that $Q \cap \partial \Omega 
\cap \Sigma \neq \emptyset$
and let $M = \#{\cE}.$  Choose a maximal subset  $\cE_0 \subset \cE$ of pairwise disjoint closed cubes.  Then $\cE_0$ has cardinality  
$\#{\cE_0} \geq c_{14}3^{-d - 1}M$ and the enlarged cubes $\frac{5}{4}Q, Q \in \cE_0$ are pairwise disjoint.  For each $Q \in \cE_0$ there exists by  (1.7) a compact set $E_Q \subset \frac{5}{4}Q \cap \partial\Omega$ and
 a ball $B(p_Q, \alpha \eta^2 2^{-j}) \subset \frac{5}{4}Q \cap \Omega$ satisfying the conclusions of Lemma 3.2 and Lemma 3.3. Now we can follow the proof Corollary 4.2 to conclude that 
$\#{\cE_0}(\eta^2 2^{-Nj})^d \leq C 2^{-Njd}$.  Hence $M \leq C\eta^{-2d}$ and there exists  a pair of adjacent closed subrings in $\cR$ whose union  meets at most $c_{15}C\eta^{1 -2d}$ dyadic cubes from $\cE.$  That implies 
(5.1).

\end{proof}

\vv

\noi {\bf {Proof of Proposition 1.3.}} For $j \geq 0$ let $V_j$ be a maximal subset 
of $\partial \Omega$ such that 
when $x, x' \in V_j$ $|x-x'| \geq 2^{-jN}$ and for $x \in V_j$ let $B_j(x)$ be the ball given by 
Lemma 5.1  and set  $\Delta_j(x) = \partial \Omega \cap B_j(x)$. Put a total order, written  $x < y$, on the finite  set $V_j$  and define 

$$\Delta_j^*(x) = \Delta_j(x) \setminus \bigcup_{y  < x}  \Delta_j(y).$$

\noi Then for each $j,$  (1.10), (1.11), and (1.12) hold for the family $\{\Delta_j^*(x)\}$  and because the balls $B(x,(1-\eta)2^{-Nj}),  x \in V_j$ are disjoint we have 

\begin{equation}\label{6.2}
B(x,(1-\eta)2^{-Nj}) \subset \Delta_j^*(x)
\end{equation}

\noi for every $x \in V_j.$  Because $\partial \Omega \subset \R^{d+1}$ there is constant $M_d$ independent of $j$ such that

\begin{equation}\label{Md}
\#\{y \in V_j: y < x ~{\rm {and}} ~ B_j(y) \cap B_j(x) \neq \emptyset\} \leq M_d.
\end{equation}

\noi Therefore by (5.1) the minimum number $m_j^*$ of closed balls $\overline {B(p,\eta^2 2^{-Nj})}$ needed to cover 

\begin{align}
E_j^*(x) &= \{y \in \Delta_j^*(x): {\rm {dist}}(y, \partial \Omega) < \eta^2 2^{-Nj} \} \nonumber \\
 & \cup \{y \in \partial \Omega \setminus \Delta_j^*(x): {\rm {dist}}(y, \Delta_j^*(x)) < \eta^2 2^{-Nj}\} \nonumber 
\end{align}

\vv
\noi has the upper bound

\begin{equation}\label{mrings}
m^*_j(x) \leq C_dM_d \eta^{1 -2d}. 
\end{equation}

\vv
However the families $\{\Delta_j^*\}_{j \geq 0}$ may not satisfy the nesting condition (1.13) or the small boundary condition 
(1.15). For those reasons we further refine each set $\Delta_j^*$, still following [Da]. 
If $x \in V_j, j \geq 1$, there exists by (1.11) and (1.12) a unique 
$\varphi(x) \in V_{j-1}$ such that  $x \in \Delta_{j-1}^*(\varphi(x)).$  For any $j$ and $x \in V_j$ define
$D_{j,0}(x) = \Delta_j^*(x)$ and for $n \in \N$,

$$D_{j,n}(x) = \bigcup\{\Delta_{j+n}^*(y): \varphi^n(y) =x\}$$

\noi   Then for any ${j,n}$

\begin{equation}\label{union1}
\bigcup\{D_{j,n}(x): x \in V_j\} = \partial \Omega
\end{equation}

\noi and by induction

\begin{equation}\label{(5.6)}
D_{j,n}(x) = \bigcup\{D_{j,n-k}(y): \varphi^k(y) = x\}.
\end{equation}

\noi for $0 \leq k \leq n.$

\vv

Write $\dist_{\HH}(A,B)$ for the Hausdorff distance between subsets $A, B$ of $\R^{d+1}.$ 
Since $\diam(\Delta_j^*) \leq(1 + \eta)2^{-Nj}$ we have $$\dist_{\HH}(D_{j,1}(x),\Delta_j^*(x)) \leq (1 + \eta)2^{-N(j+1)},$$, 

\noi so that by 
(5.6) and induction

\begin{equation}\label{(5.7)}
\dist_{\HH}(D_{j,n}(x),D_{j,n+1}(x)) \leq (1 + \eta)2^{-N(j+n)}.
\end{equation}

\vv
\noi Hence for each $j$ and $x \in V_j$ the sequence of  $\{\overline{D_{j,n}(x)}\}$ of compact sets converges in Hausdorff metric to a compact set $R_j(x).$  It is clear from (5.5) that for any fixed $j$

\begin{equation}\label{(5.8)}
\bigcup \{R_j(x): x \in V_j\} = \partial\Omega
\end{equation}

\noi because if $y \in \partial \Omega$ then $y \in D_{j,n}(x^{(n)})$ for some $x^{(n)} \in V(j)$ and because $V(j)$ is finite  there is $x \in V(j)$ with $y \in \overline {D_{j,n}(x)}$ for infinitely many $n$. 

\vv
Since we took closures (1.12) may not hold for the sets $\{R_j(x)\}$, and like [Da] we must alter them one final time. By induction we can choose the ordering on the finite set $V_j, j \geq 1$ so that $x < y$ if 
$\varphi(x) < \varphi(y).$ Then define, for all $j$ and $x \in V(j)$

\begin{equation}\label{finQ}
S_j(x) = R_j(x) \setminus \bigcup_{V(j) \ni y < x} R_j(y).
\end{equation}

\noi Then it is clear from
(5.8) that  (1.12) and (1.13) hold for the family $\SSS = \bigcup_j \{S_j\},$ 
and since by (5.7)  

\begin{equation}\label{(5.10)}
\diam(S_j(x)) \leq \diam(R_j(x)) \leq \sum_{k =j}^{\infty} 2(1+\eta)2^{-Nk} \leq 4(1+\eta)2^{-Nj}.
\end{equation}

To obtain the lower bound in (1.10) and also (1.13), (1.14) and (1.15) we need $2^{-N}$  to be small compared to $\eta$.  Assume

\begin{equation}\label{etaN}
2^{-N} \sim  \eta^2 < \frac{1}{9}.
\end{equation}  

\noi Then by (5.2) and (5.7) we have for $x \in V_j,$

\begin{align}
\dist(x,\partial \Omega \setminus D_{j,n})  &\geq (1-\eta)2^{-Nj} - \sum_{k > j} 2(1+\eta)2^{-Nk}
\nonumber \\ 
& \geq  2^{-Nj} \Bigl(1-\eta - 2(1+\eta) \frac{2^{-N}}{1 - 2^{-N}}\Bigr) \geq \frac{2^{-Nj}}{3}.
\nonumber 
\end{align}

\vv
\noi  This implies  (1.14) and  with (5.10)  it also implies (1.10). 

\vv

To show (1.13) suppose $u \in \Delta_j(x) \cap \Delta_{j+1}(y).$  Then by (5.7) $u = \lim x_n$ where 
$x_n \in V_n, x_{n+1} \in \Delta_n^*(x_n)$ and $x_j =x,$ and $u = \lim y_n$ where $y_n \in V_n, y_{n+1} \in \Delta_n^*(y_n)$ and $y_{j+1} =y.$ Hence $u \in \bigcap_{n \geq j}  R_n(x_n) \cap \bigcap_{n \geq j+1} R_n(y_n)$ so that by the definition 
(5.9)
$y_n = x_n$ for all $n \geq j+1$ and $S_{j+1}(y)  \subset S_j(x).$

\vv

To verify the small boundary condition (1.18) we can 
by (5.2) assume $\tau = 2^{-Nk},  k \geq 1.$  Let $x \in V_j$ and write $S = S_j(x).$   
Then by (5.7) and (5.10) $N_{\tau}(S)$ is comparable to 

$$\#\{y \in V_{j+k}: S_{j+k}^*(y) \cap \Delta_{\tau}(S) \neq \emptyset\},$$

\noi and by (5.4) and (5.11) this number is bounded 
by $(C_d M_d \eta^{1-2d})^k \sim (C_dM_d)^k \tau^{\frac{1}{2}}$, which,  for $C > 2$ and $\tau$ small, is bounded by 
  $C \tau^{1/C -d}$. 
\endproof 

\vv

\section{A Corona Decomposition and the Proof of Theorem 1.4 Part A.}   Assume $\Omega \subset \R^{d+1}, d \geq 1,$ is a domain satisfying (1.4), (1.7), and either {\rm (a)} or 
{\rm (b)} and let $\SSS$ be a family of subsets of $\partial \Omega$ satisfying the conclusions of Proposition 1.3.
We shall prove there exist constants $\ve_1, A_0$ and $C$ such that {\rm (1.24)} holds with constant $C$ whenever $0 < \delta < \frac{\ve}{3} < 
 \frac{\ve_1}{3}$ and 
$A > A_0$,  $S_0 \in \SSS$, and $G_k(S_0)$ are its generations defined for $\delta$ and $A$.   Recall that by Proposition  1.3 the family $\SSS$ has the properties  (1.17), (1.18), and (1.19).  

\vv

\begin{lemma}\label{lem7}  Let $S \in \SSS$ and let $\{S_j\} \subset \SSS$ be a family of cubes $S_j \subset S$ 
satisfying $S_j \cap S_k = \emptyset$ when $j \neq k.$  If 
$S_j \in \HD(S)$ for all $j$, then

\begin{equation}\label{(6.1)}
\sum \ell(S_j)^d \leq \frac{C_1}{A} \ell(S)^d,
\end{equation}

\noindent while if $S_j \in \LD(S)$ for all $j$, then

\begin{equation}\label{(6.2)}
\sup_{B_S} \sum_{S_j} \omega(p,S_j) \leq C_2 \delta,
\end{equation}

\noi where $C_1$ and $C_2$ depend only on $d$, $\delta$ and the constant in (1.12).

\end{lemma}
\vv

\begin{proof}
 Assertion (6.2) follows from  (1.20), (1.21), (1.22), (1.25) and Lemma 4.1, with constant $C_2$ depending only on 
$\delta$ and the constants in Propostion 1.3 and (1.12).    

\vv  

Since the definition of $\HD$ entails $\omega(p_S,2S_j,\Omega)$ and not $\omega(p_S,S_j,\Omega)$,
the proof of (6.1) requires more work.  Note that if $2S_k \cap 2S_j \neq \emptyset$ and 
$\ell(S_k) \leq \ell(S_j)$ then by (1.10)

 $$S_k \subset B(x_{S_j},C\ell(S_j)),$$

\noi  in which the constant $C$ depends only on the upper bound in (1.10) and thus only on $\alpha$, $\beta$ and $d$.  Hence by Theorem 1.2, Part A, 

$$\sum \bigl\{\ell(S_k)^d: 2S_k \cap 2S_j \neq \emptyset, ~~ \ell(S_k) 
\leq \ell(S_j)\bigr\} \leq C_1 \ell(S_k)^d,$$

\vv
\noi and by a Vitali argument there exists $\{S'_j\} \subset \{S_j\}$ with $2S'_j \cap 2S'_k = \emptyset$ and 

$$\sum \ell(S_j)^d \leq C_1 \sum \ell(S'_j)^d \leq \frac{C_1}{A}\sum \omega(p_S,2S'_j,\Omega) \ell(S)^d 
\leq \frac{C_1}{A} \ell(S)^d.$$

\end{proof} 

Turning to the proof of Part A of Theorem 1.4, we now assume $A  > 2C_1.$
To prove (1.29) we separate high and  low density cubes. For  $S \in \SSS$ let $GH_1(S)$ be a family of high density cubes $S' \in G_1(S)$ and  by induction

\begin{equation}\label{genH}
GH_k(S) = \bigcup_{S' \in GH_{k-1}(S)} GH_1(S').
\end {equation}

\noi  Thus if $S_k \in GH_k(S),$ then

\begin{equation}\label{highstring}
S_k \subset S_{k-1} \subset .... S_1 \subset S_0 =S
\end{equation} 

\noi in which for $j >0$ 

$$S_{j+1} \in \HD(S_j),$$

\vv

\noindent so that all ancestors of $S_k$ except possibly the first  are $\HD$ cubes.  Write 

$$GH(S) = \bigcup_{k \geq 1} GH_k(S).$$

\noi  Then
by (6.1)

\begin{equation}\label{Hsum}
\sum_{GH(S)} \ell(S')^d  = \sum_{k=1}^{\infty} \sum_{GH_k(S)} \ell(S')^d \leq \frac{C_1}{A -C_1} \ell(S)^d.
\end{equation}

Similarly, let $GL_1(S)$ be a family of low density cubes $S_j \in G_1(S)$ and  by induction

\begin{equation}\label{genL}
GL_k(S) = \bigcup_{S' \in GL_{k-1}(S)} GL_1(S').
\end {equation}

\noi  Thus if $S_k \in GL_k(S),$  then 

\begin{equation}\label{lowstring}
S_k \subset S_{k-1} \subset .... S_1 \subset S_0 =S
\end{equation}

\noi and $S_{j+1} \in \LD(S_j)$  for $j > 0.$  Write

$$GL(S) = \bigcup_{k \geq 1} GL_k(S).$$

\vv
\begin{lemma}\label{lem8}  Assume   
$\ve$ in {\rm {(1.2)}} is small and $\delta \leq \ve.$ Then there exists constant $C_2$ such that for any $S_0 \in \SSS$  

\begin{equation}\label{(6.5)}
\sum_{GL(S_0)} \ell(S)^d  = \sum_{k =1}^{\infty} \sum_{GL_k(S_0)} \ell (S)^d \leq C_2 \ell(S_0)^d.
\end{equation}

\end{lemma}

\vv
\begin{proof}  The proof is like the proof of (6.2). 
 For any $S \in GL(S)$ define

$$E_S = S \setminus \bigcup_{S' \in GL_1(S)} S'.$$ 

\noi Then $E_{S_1} \cap E_{S_2} = emptyset$ for $S_1 \neq S_2$ and that  $\inf_{B_S} \omega(p,E_S,\Omega) > 1 - \ve$  so that Lemma 4.1 and (1.13) imply (6.8).  \end{proof}

\vv
Now the proof of (1.29) follows by interlacing (6.5) and (6.8).   Write

$$L_1(S) = \sum_{GL(S)} \ell(S')^d,  ~~H_1(S) = \sum_{GH(S)} \ell(S')^d,$$

\noi and by induction

$$L_{k+1}(S) = \sum_{GH(S)} L_k(S'), ~~~H_{k+1}(S) = \sum_{GL(S)} H_k(S').$$

Then 

$$\sum_{k=1}^{\infty}\sum_{G_k(S_0)} \ell(S)^d  = \sum_{k=1}^{\infty}(L_k(S_0)  + H_k(S_0))$$

\noindent and by (6.5)  and (6.8)

$$L_{k+1}(S) \leq C_2 H_k(S)$$

$$H_{k+1}(S) \leq \frac{C_1L_k(S)}{A -C_1},$$

\vv
\noi so that writing $L_0(S) = H_0(S) =1$
and taking
$A-1 > C_1 +C_1C_2$ yields

$$\sum_{k=1}^{\infty}\sum_{G_k(S_0)} \ell(S)^d \leq \frac{AC_2 + C_1 + C_1C_2}{A -C_1 -C_1C_2}.$$

\vv
\noi That proves (1.29) and  Theorem 1.4 Part A. 
 
\vv

\section{A domain $\tilde \Omega$}

Assume  $\Omega$ is a corkscrew domain satisfying  (1.7) and  $\SSS$ is a family of subsets of $\partial \Omega$ 
 having properties (1.13) - (1.18) of  Theorem 1.3, and their consequences (1.20), (1.21) and (1.22).  Fix constants 
$\ve, \delta$, $N$, $A$ and $C$ with $0 < \delta < \frac{\ve}{3}$ and $A$ so large  that (1.27) holds
for any $S_0 \in \SSS$ when the generations $G_k(S_0)$  are define by (1.22) and (1.25). Also assume $\SSS$ satisfies the conclusion of Lemma 4.1 or, equivalently, hypothesis (ii) of Theorem 1.5. Under those assumptions  we construct a domain $\tilde \Omega \subset \Omega$ with $\partial \Omega \subset \partial \tilde \Omega$ and a $d$-Ahlfors regular measure $\sigma$ supported on  $\partial \tilde \Omega$ and  boundedly mutually absolutely continuous with 
$\chi_{\partial \tilde \Omega} \HH^d.$ 

\vv

For any $S \in \SSS$ let   
$$\Gamma_S \subset 2B_S \setminus B_S$$

\noi  be a finite union of separated closed spherical caps such that

\begin{equation}\label{sizeF}
\HH^d(\Gamma_S) = c_{16} \ell(S)^d.
\end{equation}

\noi    Since $B_S$ has diameter $2c_1 \ell(S)$  we can (and do) require   $\Gamma_S$ to be uniformly rectifiable with constants depending only on $c_0,...,c_{16}$ but not on $S$. 
Note that (taking $c_{16}$ carefully) we have  

\begin{equation}\label{starOm}
\omega(p_S,\Gamma_S, \Omega^*) \sim 1/2, 
\end{equation}

\noi for any domain $\Omega^*$ such that
$$(\Omega \setminus \Gamma_S)  \cap B(x_S, c_0 \ell(S)) \subset \Omega^* \subset \Omega.$$

\noi and by (3.4)

\begin{equation}\label{bigstar}
\omega(p_S, S \cup \Gamma_S, \Omega^*) > 1- \varepsilon  
\end{equation}

\noi for all such $\Omega^*.$
\vv
Define $\Omega_0 = \Omega$ and assume $\diam (\partial \Omega) \sim 1$ so that $\partial \Omega = S_0 \in \SSS.$ Fix $\lambda >1$ so that

\begin{equation}\label{lambda}
\lambda -1 < \dist(S,4B_S)
\end{equation}  

\noi and define

$$\widetilde \HD(S_0) = \Bigl\{S_1  \in \SSS, S_1 \subset S_0:
 \omega(p_{S_0}, \lambda (S),\Omega_0) \geq A \Bigl(\frac{\ell(S_1)}{\ell(S_0)}\Bigr)^d, ~ S_1 {\rm {maximal}}\Bigr\},$$  

$$\widetilde \LD(S_0) = \Bigl\{S_1  \in \SSS, S_1 \subset S_0: \omega(p_{S_0}, S_1,\Omega_n) \leq \delta \Bigl(\frac{\ell(S_1)}{\ell(S_0)}\Bigr)^d, S_1 ~{\rm {maximal}}\Bigr\},$$

$$\widetilde G_1 = \widetilde {G_1(S_0)} = \bigl\{S' \in \widetilde HD(S_0) \cup \widetilde \LD(S_0), S' ~~{\rm {maximal}}\bigr\},$$

$$K_1 = S_0 \setminus \bigcup_{\widetilde G_1(S_0)}S,$$ 

$$\tree(S_0) = \bigl\{S \in \SSS: S \not\subset S' ~~{\rm for~~all} ~~ S' \in {\widetilde G_1(S_0)}\bigr\},$$

$$\Omega_1 = \Omega \setminus \bigcup_{\widetilde G_1(S_0)} \Gamma_S,$$

$$\mu_1(.)  = \ell(S_0)^d \chi_{K_1} \omega(p_{S_0},.,\Omega_0),$$

$$\nu_1 = \sum_{\widetilde G_1(S_0)} \chi_{\Gamma_S}\HH^d,$$  
 
\noi and 

$$\sigma_1= \mu_1 + \nu_1.$$

\noi  Then $\sigma_1$  is a finite measure on $\partial\Omega_1$.
 
\vv

For $S \in \SSS$  define

$$S^{1} = S \cup \bigcup\{\Gamma_{S'}: S' \in {\widetilde  G_1}, S' \subset S\}$$

\noi and declare   $\ell(S^{1}) = \ell(S).$

\vv

\begin{lemma}\label{lem 9} There are constants $c_{17}$ and $c_{18}$ such that if $S \in \tree(S_0)$,

\begin{equation}\label{sig1}
c_{17} \ell(S) \leq \sigma_1(S^{1}) \leq c_{18} \ell(S).
\end{equation}

\end{lemma}

\begin{proof} For the upper bound we have
$$\mu_1(S^1) \leq A\frac{\ell(S)^d}{\ell(S_0)^d},$$ 
since $S \in \tree(S_0)$, and 
$$\nu_1(S^1) \leq C_1\ell(S)^d$$
by Lemma 4.1.

\vv

For the lower  bound note that

$$\sigma_1(S^1) = {\ell(S_0)}^d\omega(p_{S_0}, S, \Omega_0) -{\ell(S_0)}^d\sum_{{\widetilde G_1(S_0)} \ni S' \subset S} \omega(p_{S_0},S',\Omega_0)
+ \sum_{{\widetilde G_1(S_0)} \ni S' \subset S} \HH^d(\Gamma_{S'}),$$

\vv  
\noi in which 

$${\ell(S_0)}^d\omega(p_{S_0}, S, \Omega_0) \geq \delta \ell(S)^d$$

\vv

\noi while by the definition of $G_1(S_0)$

$${\ell(S_0)}^d\sum_{{\widetilde G_1(S_0)} \ni S' \subset S} \omega(p_{S_0},S',\Omega) 
\leq C_1 2^{2Nd}A \sum_{{\widetilde G_1(S_0)} \ni S' \subset S}\ell(S')^d.$$

\noi Thus if

\begin{equation}\label{(7.6)}
\sum_{{\tilde G_1(S_0)} \ni S' \subset S}\ell(S)^d \leq \frac{\delta}{C_1 2^{2Nd +1}A} \ell(S)^d
\end{equation}

\noi the lower bound holds with $c_{20} = \frac{\delta}{2}.$  On the other hand,  if (7.6) fails, 
then  $\mu_1(S^{1}) \geq 0$ and 

$$\nu_1(S^1) \geq \frac{c_{16}}{C_1 2^{2Nd +1}A} \delta.$$

\end{proof}

\vv

Now continue by induction.  
For $n \geq 1$ assume we have defined $\widetilde G_n = ~\widetilde G_n(S_0)$,  $\Omega_n$, ~ and $S^n$ for all $S \in \SSS$.
Then for each $S \in \widetilde G_n(S_0)$  define

$$\widetilde \HD(S) = \Bigl\{S_1  \in \SSS, S_1 \subset S: \omega(p_S, \lambda (S_1)^n,\Omega_n) \geq A \Bigl(\frac{\ell(S_1)}{\ell(S)}\Bigr)^d, ~ S_1 {\rm {maximal}}\Bigr\},$$

\noi   

$$\widetilde \LD(S) = \Bigl\{S_1  \in \SSS, S_1 \subset S: \omega(p_S, (S_1)^n,\Omega_n) \leq \delta \Bigl(\frac{\ell(S_1)}{\ell(S)}\Bigr)^d, ~ S_1 {\rm {maximal}}\Bigr\},$$

$$\widetilde G_1(S) = \bigl\{S' \in \tilde \HD(S) \cup \tilde \LD(S),~ S' ~~{\rm {maximal}}\bigr\}$$

$$\tree(S) = \bigl\{S' \in \SSS: S' \subset S,~~ S' \not\subset S_1~~ {\rm {for ~~all}}~~ S_1 \in {\widetilde G_1(S)}\bigr\},$$

$$\widetilde G_{n+1}(S_0) = \bigcup_{G_n(S_0)} \widetilde G_1(S),$$

$$K_{n+1} = \bigcup_{\widetilde G_n(S_0)} \bigl(S \setminus \bigcup_{\widetilde G_1(S)}S_1\bigr),$$

$$\Omega_{n+1} = \Omega_n \setminus \bigcup_{\widetilde G_{n+1}(S_0)} \Gamma_S,$$

$$\mu_{n+1}(.)  = \sum_{S  \in \widetilde G_n} \ell(S)^d \chi_{_{S \cap K_{n+1}}} \omega(p_{S},.,\Omega_n),$$

$$\nu_{n+1} = \sum_{\widetilde G_{n+1}(S_0)} \chi_{\Gamma_S}\HH^d,$$  
 
\noi and define

$$\sigma_{n+1}= \mu_{n+1} + \nu_{n+1}.$$

\vv

\noi  Then $\sigma_{n+1}$  is a finite measure on $\partial\Omega_{n+1}$.

\vv 

For $S \in \SSS$  define

$$S^{n+1} = S^{n} \cup \bigcup\bigl\{\Gamma_{S'}: S' \in \widetilde  G_{n+1}, ~ S' \subset S\bigr\}$$

\noi and  define

$$\ell(S^{n+1}) = \ell(S).$$

\noi Note that by the proof of Lemma 7.1,

\begin{equation}\label{sigman}
c_{19} \ell(S)^d \leq \sigma_{n+1}(S^{n+1}) \leq c_{20}\ell(S)^d
\end{equation}
\vv

\noi for all $S \in \tree(S'), S' \in \widetilde G_n(S_0)$.

\vv

Define 
$\tilde \Omega = \bigcap \Omega_n,$ which, as we will see, is a connected open set, and  

$$\mu = \sum_{n \geq 1}\mu_n,$$

$$\nu= \sum_{n \geq 1} \nu_n,$$

$$\sigma = \mu + \nu,$$

\noi and, for $S \in \SSS.$

$$S^{\infty} = \bigcup S^{n}.$$
\vv

\begin{lemma}\label{lem10}
Let $S \in \widetilde G_n$.  Then  

\begin{equation}\label{HD'}
\sum_{\widetilde \HD(S)} \Bigl(\frac{\ell(S_1)}{\ell(S)}\Bigr)^d \leq \frac{C_1}{A},
\end{equation}
\noi and 
\begin{equation}\label{LD'}
\sum_{\widetilde \LD(S)} \omega(p_S,S_1,\Omega) \leq C\delta + \ve,
\end{equation}
\noi where 

$$\inf_{T \in  \SSS} \inf_{p \in \Gamma_T} \omega (p,T,\Omega) \geq 1 -  \ve.$$   
\end{lemma}

\begin{proof} The proof of (7.8) is the same as the proof of (6.6) because  by Part A of Theorem 2.1  (or 
hypothesis (ii) of Part B of Theorem 1.4) the Vitali argument from that proof can still be applied.    

To prove
(7.9) let $S \in \tilde G_n$ and for $1 \leq k \leq (n-1),$ let $T_k(S)$ be that unique $T \in \tilde G_k$ such that $S \subset T_k.$  Let
$S_1 \in \tilde \LD(S).$  Then $S_1 \subset \partial \Omega \subset \partial \Omega_n$ and

$$\omega(p_S,S_1,\Omega) = \omega(p_S,S_1,\Omega_n) + \sum_{k=1}^n 
\sum_{T \in \tilde G_k \setminus \{S_1\}} \int_{\Gamma_T} \omega(p,S_1,\Omega) d\omega(p_S,p,\Omega_n).$$ 

\noi By definition and Theorem 2.1 Part A,

$$\sum_{S_1 \in \tilde\LD(S)} \omega(p_S,S_1,\Omega_n) \leq \delta \sum_{\LD(S)} 
\Bigl(\frac{\ell(S_1)}{\ell(S)}\Bigr)^d
\leq C\delta,$$

\noi while

\begin{align}
\sum_{S_1 \in \tilde\LD(S)}\sum_{k =1}^n & 
\sum_{T \in \tilde G_k \setminus \{S_1\}} \int_{\Gamma_T} \omega(p,S_1,\Omega) d\omega(p_S,dp,\Omega_n) \nonumber \\  
& = \int_{\Gamma_S} \sum_{S_1 \in \tilde\LD (S)} \omega(p,S_1,\Omega) d\omega(p_S,dp,\Omega_n)
\nonumber \\
& +\sum_{k =1}^n \sum_{T \in \tilde G_k, T \cap S = 
\emptyset}\int_{\Gamma_T} \sum_{S_1 \in \tilde \LD(S)} 
\omega(p,S_1,\Omega) d\omega(p_S,dp,\Omega_n) \nonumber \\
& +\sum_{k=1}^{n-1}\int_{\Gamma_{T_k}} \sum_{\tilde \LD(S)} \omega(p,S_1,\Omega)d\omega(p_S,dp,\Omega_n) \nonumber \\
 &= I + II + III. \nonumber  
\end{align}

\vv

By (7.2) and Harnack's inequality we have 

$$I \leq 2/3 \sum_{LD(S)} \omega(p_S,S_1,\Omega),$$

\noi and we can move term $I$ to the left side of (7.9).

\vv

 For $II$, note that 

$$(S \cup \Gamma_S) \cap \bigcup_{1 \leq k \leq n}\bigcup_{\{T \in \tilde G_k, T \cap S = \emptyset\}} \Gamma_T  = \emptyset$$

\noi so that by (7.3) we have  $II \leq  \varepsilon.$ 

\vv
 For $III$ recall that $\dist(p_{T_k}, S) \geq c_2 2^{N(n-k)} \ell(S)$.  Therefore   

$$B(x_S, c_0 \ell(S)) \cap \bigcup_{1 \leq k \leq n-1} \Gamma_{T_k} = \emptyset,$$

\noi  so that by (1.23) $III < C\ve.$

\vv
   That established (7.9)  and Lemma 7.2.
\end{proof}

\vv

If $C\delta + \varepsilon$ is small, Lemma 7.2 and the proof of Lemma 6.2 yield

\begin{equation}\label{basest}
\sum_{k=1}^{\infty} \sum_{\tilde G_k} \Bigl(\frac{\ell(S_1)}{\ell(S)}\Bigr)^d \leq C_3
\end{equation}
\noi for any $S \in \SSS$.

\vv

By (7.1) and (7.10)  $\tilde \Omega = \bigcup \tilde \Omega_n$ is a connected  open set and 

$$\partial \tilde\Omega = \partial \Omega \cup \bigcup_{n =1}^{\infty} \bigcup_{S  \in \tilde G_n} \Gamma_S.$$

\noi By (7.7) $\sigma$ is a finite measure on $\partial \tilde \Omega$ such that for all $S \in \SSS$

$$c_{21} \ell(S)^d \leq \sigma(S^{\infty}) \leq c_{22} \ell(S)^d$$

\noi and by Lemma 7.1 and the definition of $\nu_{n+1}$ 

$$\sigma(E) = \HH^d(E)$$

\noi for all Borel $E \subset \bigcup \Gamma_S.$  In view of properties (1.13) and (1.17) of $\SSS$, these imply that 
$\sigma$ is a $d-$Ahlfors regular measure with closed support $\partial \tilde \Omega$
and hence that $\partial \tilde \Omega$ is $d$-Ahlfors regular.
Moreover, the family

$$\SSS^{\infty} = \bigcup_{S \in \SSS} S^{\infty} \cup \bigcup_{S \in \cup_n{\tilde G_n}}\FF_S,$$

\noi where $\FF_S$ is the dyadic decomposition of $\Gamma_S$ in spherical coordinates, is a family of Christ-David cubes for $\partial \tilde\Omega.$ 

\vv
\section{Proof of Theorem 1.1 Part B.}

To prove Theorem 1.2 Part B we assume $\Omega$ is a corkscrew domain satisfying (1.4) and either {\rm (a)} or {\rm (b)} and we let 
$\tilde \Omega$ be the domain constructed from $\Omega$ in Section 7. Recall that $\partial {\tilde \Omega}$
is $d$-Ahlfors regular.  We will prove  $\partial {\tilde \Omega}$ is uniformly rectifiable by repeating the proof of Lemma 6.2 and applying Proposition 5.1 of [GMT]. Define $G^*_0(S_0^{\infty}) = \{S_0^{\infty}\}$ and by induction, for $S^{\infty} \in G^*_n$
 define

$$\HD(S^{\infty}) = \Bigl\{S_1^{\infty}  \in \SSS^{\infty}: S_1^{\infty} \subset S^{\infty}, \omega(p_S,
\lambda S_1^{\infty},\tilde \Omega) \geq A \Bigl(\frac{\ell(S_1)}{\ell(S)}\Bigr)^d, 
S_1^{\infty} ~{\rm{maximal}}\Bigr\},$$

$$\LD(S^{\infty}) = \Bigl\{S_1^{\infty}  \in \SSS^{\infty}: S_1^{\infty} \subset S^{\infty}, 
\omega(p_S,S_1^{\infty},\tilde \Omega) \leq \delta \Bigl(\frac{\ell(S_1)}{\ell(S)}\Bigr)^d, S_1^{\infty} ~ 
{\rm {maximal}}\Bigr\},$$

$${G^*_1(S^{\infty})} = \bigl\{S_1^{\infty} \in \HD(S^{\infty}) \cup \LD(S^{\infty}), S_1^{\infty} ~~{\rm maximal}\bigr\},$$

$$\tree(S^{\infty}) = \bigl\{S_1^{\infty} \in \SSS: S_1^{\infty} \subset S^{\infty}, ~~S_1^{\infty} \not\subset S_2^{\infty} ~~{\rm {for~~ all}}~~ S_2^{\infty} \in {G^*_1(S^{\infty})}\bigr\},$$  

\noi and  

$$G^*_{n+1} = \bigcup_{S^{\infty} \in G_n^*} G^*_1(S^{\infty}).$$

\vv

\begin{lemma}\label{lem11}
Let $S^{\infty} \in G^*_n$.  Then  

\begin{equation}\label{HD*}
\sum_{S_1^{\infty} \in \HD(S^{\infty})} \Bigl(\frac{\ell(S_1)}{\ell(S)}\Bigr)^d 
\leq \frac{C_1}{A},
\end{equation}
\noi and 
\begin{equation}\label{LD*}
\sum_{S_1^{\infty} \in \LD(S^{\infty})} \omega(p_S,S_1,\Omega) \leq C\delta + \ve,
\end{equation}
\noi where 

$$\inf_{T \in  \SSS} \inf_{p \in \Gamma_T} \omega (p,T,\Omega) \geq 1 -  \ve.$$   
\end{lemma}

\medskip
\begin{proof} The proof of (8.1) is the same as the proof of (6.8).  
To prove 
(8.2) we follow the proof of (6.9) and (7.9).  Let $S_1^{\infty} \in \LD(S^{\infty}).$  Then 

$$\omega(p_S,S_1,\Omega) \leq \omega(p_S,S_1^{\infty},\tilde \Omega) + \sum_{k \geq 1}
\sum_{G_k^* \setminus \{S_1\}} \int_{\Gamma_T} \omega(p,S_1,\Omega) d\omega(p_S,p,\tilde \Omega).$$ 

\noi By definition and Theorem 1.2, Part A, 

\begin{equation}\label{(8.3)}
\sum_{S_1^{\infty} \in \LD(S^{\infty})} \omega(p_S,S_1^{\infty},\tilde\Omega) 
\leq \delta \sum_{\LD(S^{\infty})} \Bigl(\frac{\ell(S_1^{\infty})}{\ell(S^{\infty})}\Bigr)^d
\leq C\delta,
\end{equation}

\noi and 


\begin{align}
 \sum_{S_1^{\infty} \in \LD(S^{\infty})} & ~~ \sum_{k =1}^{\infty} ~~ 
\sum_{T \in G^*_k \setminus \{S_1\}} 
\int_{\Gamma_T} \omega(p,S_1,\Omega) d\omega(p_S,dp,\tilde\Omega) \nonumber \\
 & =  \int_{\Gamma_S} \sum_{S_1^{\infty} \in \LD (S^{\infty})} \omega(p,S_1,\Omega) d\omega(p_S,dp,\tilde \Omega) \nonumber \\ 
 & + \sum_{k =1}^{\infty} ~~ \sum_{T \in G^*_k, T \cap S = \emptyset} ~~ \int_{\Gamma_T} \sum_{S_1^{\infty} \in \LD(S^{\infty})} 
\omega(p,S_1,\Omega) d\omega(p_S,p,\tilde \Omega) \nonumber \\ 
& +\sum_{k=1}^{n-1}\int_{\Gamma_{T_k}} 
\sum_{S_1^{\infty} \in \LD(S^{\infty})}\omega(p,S_1,\Omega)d\omega(p_S,p,\tilde \Omega) \nonumber \\ 
& +\sum_{S_1 \in G^*_1(S)} ~~
\sum_{T \in \bigcup_k  G^*_k(S_1)} ~~\int_{\Gamma_T} \omega(p,S_1,\Omega) 
d\omega(p_S, p, \tilde \Omega) \nonumber \\
& = I' + II' + III' + IV'. \nonumber  
\end{align}

\noi Here $I'$, $II'$ and $III'$ can be handled the same way as $I$, $II$, and $III$ were, while 
$IV' \leq C\varepsilon$ by (8.3). 
\end{proof}

Thus if  $C\delta +  \varepsilon$ is small, Lemma 7.2 and  Lemma 6.3 yield

\begin{equation}\label{basest*}
\sum_{k=1}^{\infty} \sum_{G^*_k} \Bigl(\frac{\ell(S_1)}{\ell(S)}\Bigr)^d \leq C_3
\end{equation}

\noi for any $S^{\infty} \in \SSS^{\infty}$ and  any $S^{\infty}_1 \in \tree(S^{\infty})$,

\begin{equation}\label{8.5}
\delta \Bigl(\frac{\ell(S_1)}{\ell(S)}\Bigr)^d \leq \omega(p_S, \lambda S_1^{\infty},\tilde \Omega)
\leq A\Bigl(\frac{\ell(S_1}{\ell(S)}\Bigr)^d. 
\end{equation}
 
 By (8.5) and  Proposition 5.1 of [GMT] this proves $\partial \tilde \Omega$ is  uniformly rectifiable,  and that establishes Part B of Theorem 1.1.    

\vv

\section {Proof of  Theorem 1.4 Part B and Theorem 1.2 Part B.}

To prove Part B of Theorem 1.4 note that under its hypotheses the arguments in Section 7 and Section 8 show that the constructed domain $\tilde \Omega$ has  uniformly rectifiable boundary. Therefore by Part A of Theorem 1.1, {\rm (a)} and {\rm (b)} hold for
$\Omega$. 

\vv
To prove Part B of Theorem 1.2 note that its hypotheses imply Proposition 1.3  and hence  condition (ii) of Part B of Theorem 1.4. Then the argument in Section 6 yields (1.29), so that  Part B of Theorem 1.4 implies Part B of Theorem 1.2.












\vv




\end{document}